\documentclass[12pt]  {amsart} 
\usepackage{amsmath,latexsym,amssymb,amsmath,
amscd,amsthm,amsxtra}
\usepackage{txfonts}
\usepackage{color}

\newtheorem{theorem}{Theorem}[section]

\newtheorem{lemma}{Lemma}[section]

\def\k{\kappa}
\def\R{\mathbb R}
\def\sn{\mathrm{sn}}
\def\Ric{\mathrm{Ric}}
\def\Li{\mathrm{Lip}}

\def\<{\langle}
\def\>{\rangle}

\begin{document}

\title{A note on local gradient estimate on  Alexandrov spaces}\let\thefootnote\relax\footnote{ 
2010 \textit{Mathematics Subject Classification}.
Primary 51F99; Secondary 51K10, 53C21,31C05.
}
\let\thefootnote\relax\footnote{ 
\textit{Key words and phrases}. Alexandrov spaces, harmonic
functions, Yau's gradient estimate. }

\author{Bobo Hua and Chao Xia}
\address{Max-Planck-Institut f\"ur Mathematik in den Naturwissenschaft, Inselstr. 22, D-04103, Leipzig, Germany}
\email{bobo.hua@mis.mpg.de,  chao.xia@mis.mpg.de}


\maketitle


\begin{abstract} In this note, we prove Cheng-Yau type local gradient
estimate for harmonic functions on  Alexandrov spaces with Ricci
curvature bounded below. We adopt a refined version of Moser's
iteration which is based on Zhang-Zhu's Bochner type formula in
\cite{ZZ3}. Our result improves the previous one of Zhang-Zhu
\cite{ZZ3} in the case of negative Ricci lower bound.
\end{abstract}


\section{Introduction}

In 1975, Yau \cite{Y} proved that complete Riemannian manifolds with
nonnegative Ricci curvature have the Liouville property. Later,
Cheng-Yau \cite{CY} proved the following local version of Yau's
gradient estimate.

\

\noindent{\bf Theorem A} (Yau \cite{Y}, Cheng-Yau \cite{CY}). {\it
Let $M^n$ be an n-dimensional complete noncompact Riemannian
manifold with Ricci curvature bounded from
below by $-K$ $ (K \geq 0)$. Then there exists a
constant $C=C(n)$, depending only on n, such that every positive
harmonic function $u$ on geodesic ball $B_{2R}\subset M$  satisfies
\begin{eqnarray*}
\frac{|\nabla  u|}{u}\leq C\frac{1+\sqrt{K}R}{R} \ \ \ \ \
\mathrm{in}\ B_R.
\end{eqnarray*}}

\

To prove the regularity of harmonic functions on Alexandrov spaces is a
challenging problem because of the lack of the smoothness of the
metrics; the H\"older continuity of harmonic functions is well-known
(see e.g. Kuwae-Machigashira-Shioya \cite{KMS}). In 1996, Petrunin
\cite{P} proved the Lipschitz continuity of harmonic functions on
Alexandrov spaces. Recently, Zhang-Zhu \cite{ZZ1} introduced a
notion of Ricci curvature on Alexandrov spaces. Using a delicate
argument initiated by Petrunin, Zhang-Zhu \cite{ZZ3} proved the
Bochner formula on Alexandrov spaces which gives a quantitative
estimate, i.e., Yau's gradient estimate for harmonic functions.

\

\noindent{\bf Theorem B} (Zhang-Zhu \cite{ZZ3}). {\it Let $X$ be an
$n$-dimensional Alexandrov space with Ricci curvature bounded from
below by $-K$ $(K\geq 0)$, and let $\Omega$ be a bounded domain in
$X$. Then there exists a constant $C=C(n,
\sqrt{K}\mathrm{diam}(\Omega))$, depending on $n$ and
$\sqrt{K}\mathrm{diam}(\Omega)$, such that every positive harmonic
function $u$ on $\Omega$ satisfies
\begin{eqnarray}
\frac{|\nabla  u|}{u}\leq C\frac{1+\sqrt{K}R}{R} \ \ \ \ \
\mathrm{in}\ B_R,
\end{eqnarray}
for any geodesic ball $B_{2R}\subset \Omega$. If $K=0$, the constant $C$
depends only on $n$. }

\

For the case $K>0$ Theorem B is not satisfactory, compared with
Theorem A, since the constant $C$ depends not only on the dimension
$n$ but also on $\sqrt{K}R$. In this note, we refine the argument of
Zhang-Zhu \cite{ZZ3} and
derive a local gradient estimate analogous to the Riemannian case.
Our main result is the following.

\begin{theorem}\label{main thm}
Let $X$ be an $n$-dimensional Alexandrov space with Ricci curvature
bounded from below by $-K$ $(K> 0)$. Then there exists a constant
$C=C(n)$, depending only on $n,$ such that every positive harmonic
function $u$ on geodesic ball $B_{2R}\subset M$  satisfies
\begin{eqnarray*}
\frac{|\nabla  u|}{u}\leq C\frac{1+\sqrt{K}R}{R} \ \ \ \ \
\mathrm{in}\ B_R.
\end{eqnarray*}
\end{theorem}

\

Owing to the lack of regularity of harmonic functions on Alexandrov
spaces, one cannot use the method of maximum principle that was used
in Yau \cite{Y} and Cheng-Yau \cite{CY}. As in Zhang-Zhu \cite{ZZ3},
we start with the Bochner formula established in \cite{ZZ3} and use
Moser's iteration argument. For Alexandrov spaces with negative
Ricci curvature, we refine a local uniform Sobolev inequality (see
Theorem \ref{Sob}) and obtain an $L^p$ estimate of $|\nabla u|^2$
for $p\sim 1+\sqrt{K}R;$ this estimate is a starting point of
Moser's iteration. This is adapted from the idea of Wang-Zhang
\cite{WZ}. The similar idea has been successfully applied to Finsler
manifolds by the second author \cite{X}.

\

The rest of the paper is organized as follows: In Section 2, we recall
some basic and known results on Alexandrov spaces. In
Section 3, we prepare the analytic tool , i.e., Sobolev inequality for
Moser's iteration. Section 4 is devoted to the proof of
Theorem \ref{main thm}.

\section{Preliminaries on Alexandrov spaces}
A metric space $(X,d)$ is called an Alexandrov space if it is a
complete locally compact geodesic space with sectional curvature
bounded below locally in the sense of Alexandrov, i.e., locally
satisfying Toponogov's triangle comparison and of finite Hausdorff
dimension. We refer to \cite{BBI,BGP} for the basic facts of
Alexandrov spaces. It is well-known that the Bishop-Gromov volume
comparison holds on Alexandrov spaces.

Kuwae-Shioya \cite{KS1,KS2,KS3,KS4} introduced and investigated a
notion of infinitesimal Bishop-Gromov volume comparison, $BG(n,\k)$.
Let $(X,d)$ be an $n$-dimensional Alexandrov space. Set for any
$\k \in \R,$
\[\sn_{\k}(t)=\left\{\begin{array}{ll}
\frac{\sin(\sqrt\k t)}{\sqrt{\k}}, & \k>0,\\
t,& \k =0,\\
  \frac{\sinh(\sqrt{|\k|}t)}{\sqrt{|\k|}}, &\k<0.
\end{array}\right.\]
For any $p\in X,$ denote by $\hbox{r}_p(x):=d(x,p)$ the distance
function from $p.$ For $p\in X$ and $0<t\leq 1,$ we define a subset
$W_{p,t}\subset X$ and a map $\Phi_{p,t}: W_{p,t}\to X$ as follows:
$x\in W_{p,t}$ if and only if there exists some $y\in X$ such that
$x\in py$ and $\hbox{r}_p(x):\hbox{r}_p(y)=t:1,$ where $py$ is a
minimal geodesic (shortest path) from $p$ to $y.$ For any $x\in
W_{p,t},$ such $y$ is unique (by Toponogov's triangle comparison)
and we define $\Phi_{p,t}(x)=y.$ Let $\mathcal{H}^n$ denote the
$n$-dimensional Hausdorff measure on $(X,d).$ The infinitesimal
Bishop-Gromov condition for $X$ with the measure $\mathcal{H}^n$,
$BG(n,\k)$, is defined as follows: For any $p\in X$ and $t\in(0,1],$
we have
\begin{equation}
d((\Phi_{p,t})_*\mathcal{H}^n)(x)\geq \frac{t\sn_{\k}(t \hbox{r}_p(x))^{n-1}}{\sn_{\k}(
\hbox{r}_p(x))^{n-1}}d\mathcal{H}^n(x)
\end{equation} for any $x\in X$ ($\hbox{r}_p(x)<\pi/{\sqrt{\k}}$ if $\k>0),$ where $(\Phi_{p,t})_*\mathcal{H}^n$ is
the push-forward of the measure $\mathcal{H}^n$ by $\Phi_{p,t}$. In
Riemannian case, the condition $BG(n,\k)$ is equivalent to $\Ric\geq
(n-1)\k.$ This definition reflects a key property of the Ricci
curvature on Riemannian manifolds, i.e., the infinitesimal volume
comprison. There exists another notion of volume comparison called
$MCP(n,\k)$, which was defined by Sturm \cite{St} and Ohta
\cite{O1}. $MCP(n,\k)$ is equivalent to $BG(n,\k)$ in the Alexandrov
space (see \cite{O1, KS3}).

The infinitesimal version of Bishop-Gromov volume comparison implies
the global version of Bishop-Gromov volume comparison, which is also
called relative volume comparison (see \cite{KS3}). Denote by
$B_R(p):=\{x\in X; d(x,p)<R\}$ the geodesic ball of radius $R$
centered at $p,$ by $|B_R(p)|:=\mathcal{H}^n(B_R(p))$ the volume of
the ball $B_R(p)$ and by $V^{n,\k}_r$ the volume of a geodesic ball
of radius $r$ in the space form $\Pi^{n,\k},$ i.e., the complete
simply connected $n$-dimensional Riemannian manifold with constant
sectional curvature $\k.$

\begin{theorem}[\cite{KS3}]
Let $X$ be an Alexandrov space satisfying $BG(n,\k),$ $\k\leq 0.$
Then for any $p\in X$ and $0<r<R,$ we have
\begin{equation}\label{BGV}
\frac{|B_R(p)|}{|B_r(p)|}\leq \frac{V_R^{n,\k}}{V_r^{n,\k}}\leq
e^{2\sqrt{-(n-1)\k}R}\left(\frac{R}{r}\right)^n.
\end{equation}
\end{theorem}

\

Zhang-Zhu \cite{ZZ1,ZZ2} introduced a geometric version of lower
Ricci curvature bounds on Alexnadrov spaces, denoted by $\Ric\geq
-K$ ($K\geq0$). This definition reflects another important feature
of Ricci curvature in Riemannian case, the Bochner formula. For
details, we refer to \cite{ZZ3}. It was proved in \cite{ZZ1} that
$\Ric\geq -K$ for an $n$-dimensional Alexandrov space implies
Lott-Villani-Sturm's curvature dimension condition $CD(n,K)$ (for
definition, see \cite{LV1, LV2, St}) and Kuwae-Shioya's
infinitesimal Bishop-Gromov comparison $BG(n,-K/{(n-1)}).$

We recall some basic results on Alexandrov spaces. For a domain
$\Omega\subset X,$ we denote by $\Li(\Omega)$ ($\Li_0(\Omega)$) the
space of (compact supported) Lipschitz functions on $\Omega.$ It can
be shown that every Lipschitz function is differentiable
$\mathcal{H}^n$-almost everywhere and has the bounded gradient (see
Cheeger \cite{C}). For a precompact domain $\Omega'\subset\subset X$
and $u\in \Li(\Omega')$, the $W^{1,2}$ norm of $u$ is  defined as
$$\|u\|_{W^{1,2}(\Omega')}^2=\int_{\Omega'}u^2+\int_{\Omega'}|\nabla
u|^2.$$ Here, each integral above means the integration with respect
to $\mathcal{H}^n$. The space $W^{1,2}(\Omega')$ (resp.
$W^{1,2}_0(\Omega')$) is the completion of $\Li(\Omega')$ (resp.
$\Li_0(\Omega')$) with respect to the $W^{1,2}$ norm defined above.
For a domain $\Omega\subset X,$ the local $W^{1,2}$ space of
$\Omega$, $W^{1,2}_{\mathrm{loc}}(\Omega),$ consists of functions
$u$ with $u|_{\Omega'}\in W^{1,2}(\Omega')$ for any
$\Omega'\subset\subset \Omega.$ For $u\in
W^{1,2}_{\mathrm{loc}}(\Omega),$ we define a linear functional
$\mathcal{L}_u$ on $\Omega$ (corresponding to $\Delta u$ in the
smooth setting) by
$$\mathcal{L}_u(\eta)=-\int_{\Omega} \langle\nabla u, \nabla \eta\rangle\ \ \ \ \mathrm{for}\ \eta\in
\Li_0(\Omega).$$ In general, $\mathcal{L}_u$ is a signed Radon
measure on $\Omega$. Let $f\in L^2(\Omega).$ We then say a function
$u\in W^{1,2}_{\mathrm{loc}}(\Omega)$ solves the Poisson equation
$\mathcal{L}_u=f\cdot \mathcal{H}^n$ on $\Omega$ if
$$\mathcal{L}_u(\eta)=\int_{\Omega} f\cdot \eta \ \ \ \ \mathrm{for}\ \ \eta\in
\Li_0(\Omega).$$ Zhang-Zhu \cite{ZZ3} established the following
Bochner formula on Alexandrov spaces. For simplicity, we only
formulate it in the following special case, where we choose
$f(x,s)=-s$ in \cite[Theorem~1.2]{ZZ3}, since it is sufficient for
our purpose.

\begin{theorem}[Zhang-Zhu \cite{ZZ3} Theorem 1.2]  Let $\Omega$ be a domain in an
$n$-dimensional Alexandrov space $X$ with $\Ric\geq -K.$
let $u\in \Li(\Omega)$ solve the Poisson equation
$$\mathcal{L}_u=-|\nabla u|^2\cdot \mathcal{H}^n.$$ Then $|\nabla u|^2\in
W^{1,2}_{\mathrm{loc}}(\Omega)$ and
\begin{equation}\label{BOCH}
\frac{1}{2}\mathcal{L}_{|\nabla u|^2} \geq
\left(\frac{1}{n}|\nabla u|^4-\langle\nabla u, \nabla |\nabla u|^2\rangle-K |\nabla
u|^2\right)\cdot \mathcal{H}^n.
\end{equation}
\end{theorem}

\

\section{Local uniform Poincar\'e inequality and  Sobolev inequality}

The Sobolev inequality is necessary to carry out Moser's iteration.
In view of the standard theory for general metric measure spaces,
one needs the volume doubling condition and the local  uniform
Poincar\'e inequality to prove the local Sobolev inequality. In
fact, for an Alexandrov space with $\Ric\geq -K$ there are stronger
volume growth properties, i.e., infinitesimal Bishop-Gromov volume
comparison $BG(n,-K/(n-1))$ and the global version of \eqref{BGV}. 
These volume growth properties of Alexandrov spaces with Ricci
curvature bounded below are sufficient to prove the Poincar\'e
inequality.

Since the argument for proving the Poincar\'e inequality is  now
standard, we  give only a sketch  here  (see e.g.
\cite[Theorem~5.6.5]{SC}, \cite[Theorem~7.2]{KMS},
\cite[Theorem~3.1]{Hua}). By the infinitesimal Bishop-Gromov volume
comparison and the change of variables, one can show the so-called
weak Poincar\'e inequality, i.e., for any $u\in W_{\rm
loc}^{1,2}(B_{2R}),$ $$\int_{B_R} |u-\bar{u}|^2 \leq
Ce^{C\sqrt{K}R}R^2\int_{B_{2R}} |\nabla u|^2,$$ where $C=C(n)$ and
$\bar{u}=\frac{1}{|B_R|}\int_{B_R} u.$ The only difference of the weak Poincar\'e inequality from  the
desired one \eqref{wPo} below is that the integration in
the right-hand side is over $B_{2R}$ instead of $B_R.$ A
Whitney-type argument, which relies on the doubling property of the
measure, yields the Poincar\'e inequality from the weak one. None of
these arguments involve the smoothness assumptions of the metric,
hence adaptable to our setting. We remark that the precise constant,
$e^{C\sqrt{K}R}$, in the Poincar\'e inequality is crucial for this
paper.


\begin{lemma}[local uniform Poincar\'e inequality]
Let $X$ be an $n$-dimensional Alexandrov space with Ricci curvature
bounded from below by $-K$ $(K>0)$.  Then there exists $C=C(n)$ such
that for $B_{R}\subset X$ and $u\in W_{\rm loc}^{1,2}(B_R)$,
\begin{eqnarray}\label{wPo}
\int_{B_R} |u-\bar{u}|^2 \leq Ce^{C\sqrt{K}R}R^2\int_{B_{R}} |\nabla
u|^2,
\end{eqnarray} where $\bar{u}=\frac{1}{|B_R|}\int_{B_R} u.$
\end{lemma}

As long as  the uniform local Poincar\'e inequality and the
Bishop-Gromov volume comparison \eqref{BGV} are applicable, one can
obtain the next local uniform Sobolev inequality by the same
argument as in \cite[Lemma~3.2]{MW} (see also \cite{HK}).
\begin{theorem}[local uniform Sobolev inequality]\label{Sob}Let $X$ be an $n$-dimensional Alexandrov space with Ricci curvature bounded from below by $-K$ $(K>0)$. Then there exist two constants $\nu>2$ and $C$, both depending only on $n$, such that for $B_{R}\subset X$ and $u\in W_{\rm loc}^{1,2}(B_R)$,
\begin{eqnarray}\label{Sobolev0}
\left(\int_{B_R} (u-\bar{u})^{\frac{2\nu}{\nu-2}}
\right)^{\frac{\nu-2}{\nu}} \leq
e^{C(1+\sqrt{K}R)}R^2|B_R|^{-\frac2\nu}\int_{B_R}|\nabla
u|^2,\end{eqnarray}  where
$\bar{u}=\frac{1}{|B_R|}\int_{B_R} u .$ In particular,
\begin{eqnarray}\label{Sobolev}
\left(\int_{B_R} u^{\frac{2\nu}{\nu-2}} \right)^{\frac{\nu-2}{\nu}}
\leq e^{C(1+\sqrt{K}R)}R^2|B_R|^{-\frac2\nu}\int_{B_R} (|\nabla
u|^2+R^{-2}u^2).\end{eqnarray}
\end{theorem}

\

\section{Proof of Theorem \ref{main thm}}
Without loss of generality, we may assume that $u$ is a positive
harmonic function on $B_{4R}$. It was proved in \cite{P} and
\cite{ZZ3} that $u$ is locally Lipschitz continuous in $B_{4R}$,
$|\nabla u|$ is lower semi-continuous in $B_{4R}$
 and $|\nabla u|^2\in W_{\rm loc}^{1,2}(B_{4R})$. Denote $v=\log u$. One can easily verify that 
\begin{eqnarray}\label{Peq1}
\mathcal{L}_v=-|\nabla v|^2\cdot \mathcal{H}^n.
\end{eqnarray}

Since $v\in \Li(B_{2R})$, by setting $f=|\nabla v|^2$, it follows
from the Bochner formula \eqref{BOCH} that for any $0\leq\eta\in \Li_0(B_{2R})$,
\begin{eqnarray}\label{Peq2} \int_{B_{2R}} \langle\nabla\eta, \nabla f\rangle  \leq \int_{B_{2R}} \eta \left(2\langle \nabla v, \nabla f\rangle+2Kf-\frac{2}{n}f^2\right).\end{eqnarray}

In fact, by an approximation argument, \eqref{Peq2} holds for any
$0\leq\eta\in W^{1,2}_0(B_{2R})\cap L^{\infty}(B_{2R}).$ Let
$\eta=\phi^2f^\beta$, with $\phi\in \Li_0(B_{2R})$, $0\leq\phi\leq 1$ and $\beta\geq
1$. 
Then $\eta$
is an admissible test function for \eqref{Peq2}. Hence we have from
\eqref{Peq2} that
\begin{eqnarray*}\label{Peq3}
&&\int_{B_{2R}} \beta\phi^2f^{\beta-1}|\nabla f|^2+2\phi f^\beta
\langle\nabla f,\nabla \phi\rangle \nonumber\\&\leq &\int_{B_{2R}}
\phi^2f^\beta \left(2\langle \nabla v, \nabla
f\rangle+2Kf-\frac{2}{n}f^2\right) .
\end{eqnarray*}

It follows that
\begin{eqnarray*}\label{Peq7}
\frac{4\beta}{(\beta+1)^2}\int_{B_{2R}} \phi^2 |\nabla
f^{\frac{\beta+1}{2}}|^2  &\leq&\frac{4}{\beta+1}\int_{B_{2R}} \phi
f^{\frac{\beta+1}{2}}|\nabla\phi|  |\nabla
f^{\frac{\beta+1}{2}}|\nonumber\\&&+\frac{4}{\beta+1}\int_{B_{2R}}
\phi^2f^{\frac{\beta+2}{2}}|\nabla
f^{\frac{\beta+1}{2}}|\nonumber\\&&-\int_{B_{2R}}
\frac{2}{n}\phi^2f^{\beta+2}+\int_{B_{2R}} 2K\phi^2f^{\beta+1}.
\end{eqnarray*}
Using the H\"older inequality, we obtain
\begin{eqnarray*}\label{Peq8}
\int_{B_{2R}} \phi^2 |\nabla f^{\frac{\beta+1}{2}}|^2 &\leq
&C_1\int_{B_{2R}} |\nabla\phi|^2  f^{\beta+1}+C_2\int_{B_{2R}}
\phi^2f^{\beta+2}\nonumber\\&&-C_3 \beta\int_{B_{2R}}
\phi^2f^{\beta+2}+C_4\beta K\int_{B_{2R}} \phi^2f^{\beta+1}.
\end{eqnarray*}
We remark that from now on, constants $C_1, C_2, \ldots$ depend only
on $n$.

For $\beta\geq 2C_2/C_3$, we have
\begin{eqnarray}\label{Peq8}
&&\int_{B_{2R}}  |\nabla (\phi
f^{\frac{\beta+1}{2}})|^2+\frac12C_3\beta\int_{B_{2R}} \phi^2f^{\beta+2}
\nonumber\\&\leq &2C_1\int_{B_{2R}} |\nabla\phi|^2
f^{\beta+1}+C_4\beta K\int_{B_{2R}} \phi^2f^{\beta+1}.\end{eqnarray}
Using the Sobolev inequality \eqref{Sobolev}, we obtain
\begin{eqnarray}\label{Peq15}
\left(\int_{B_{2R}}\phi^{2\chi} f^{(\beta+1)\chi}
\right)^{\frac{1}{\chi}} & \leq&
e^{C_5(1+\sqrt{K}R)}R^2|B_{2R}|^{-\frac2\nu}\bigg(C_6\int_{B_{2R}}
|\nabla \phi|^2 f^{\beta+1}\nonumber\\&&+(C_7\beta
K+C_8R^{-2})\int_{B_{2R}} \phi^2f^{\beta+1}-\beta\int_{B_{2R}}
\phi^2f^{\beta+2}\bigg),\end{eqnarray} where $\chi=\nu/(\nu-2)$.

\

We first use \eqref{Peq15} to prove the following:
\begin{lemma}\label{lem1}There exists a large positive constant $C_9$ and $C_{10}$ such that  for $\beta_0=C_{10}(1+\sqrt{K}R)$ and $\beta_1=(\beta_0+1)\chi$,  we have $f\in L^{\beta_1}(B_{\frac32R})$ and \begin{eqnarray}\label{Peq10}
\|f\|_{L^{\beta_1}\big(B_{\frac32R}\big)}  &\leq &C_9
\frac{(1+\sqrt{K}R)^2}{R^2}|B_{2R}|^{\frac{1}{\beta_1}}.\end{eqnarray}
\end{lemma}

\begin{proof}
Let $C_{10}\geq 2C_2/C_3$ such that $\beta_0=C_{10}(1+\sqrt{K}R)$
satisfies \eqref{Peq8} and \eqref{Peq15}. We rewrite \eqref{Peq15}
for $\beta=\beta_0$ as
\begin{eqnarray}\label{Qeq0}
\left(\int_{B_{2R}}\phi^{2\chi} f^{(\beta_0+1)\chi}
\right)^{\frac{1}{\chi}}& \leq
&e^{C_{11}\beta_0}|B_{2R}|^{-\frac2\nu}\bigg(C_6R^2\int_{B_{2R}} |\nabla
\phi|^2 f^{\beta_0+1}\\&&+C_{12}\beta_0^3\int_{B_{2R}}
\phi^2f^{\beta_0+1}-\beta_0R^2\int_{B_{2R}}
\phi^2f^{\beta_0+2}\bigg).\nonumber\end{eqnarray}

 We estimate the second term in the right-hand side of \eqref{Qeq0} as
 follows:
\begin{eqnarray}\label{Qeq1} C_{12}\beta_0^3\int_{B_{2R}} \phi^2f^{\beta_0+1}&=&
C_{12}\beta_0^3\left(\int_{\{f\geq 2C_{12}\beta_0^2R^{-2}\}} \phi^2f^{\beta_0+1}+\int_{\{f< 2C_{12}\beta_0^2R^{-2}\}} \phi^2f^{\beta_0+1}\right)\nonumber\\
&\leq &\frac12\beta_0R^2 \int_{B_{2R}}
\phi^2f^{\beta_0+2}+C_{13}^{\beta_0+1}\beta_0^3\left(\frac{\beta_0}{R}\right)^{2(\beta_0+1)}|B_{2R}|.\end{eqnarray}
Set $\phi=\psi^{{\beta_0+2}}$ with $\psi\in \Li_0(B_{2R})$
satisfying
\begin{eqnarray*}\label{Peq11}
0\leq \psi\leq 1,\quad \psi\equiv 1 \hbox{ in } B_{\frac32R},\
|\nabla \psi|\leq \frac{C_{14}}{R}.\end{eqnarray*} Then
$$R^2|\nabla\phi|^2\leq
C_{15}\beta_0^2\phi^{\frac{2(\beta_0+1)}{\beta_0+2}}.$$ By the
H\"older inequality and the Young inequality, the first term in the
right-hand side of \eqref{Qeq0} can be estimated as follows:
\begin{eqnarray}\label{Qeq2}
C_6R^2\int_{B_{2R}} |\nabla\phi|^2 f^{\beta_0+1}
&\leq& C_{16}\beta_0^2\int_{B_{2R}} \phi^{\frac{2(\beta_0+1)}{\beta_0+2}}f^{\beta_0+1}\nonumber\\
&\leq &C_{16}\beta_0^2\left(\int_{B_{2R}} \phi^{2}f^{\beta_0+2}\right)^{\frac{\beta_0+1}{\beta_0+2}}|B_{2R}|^{\frac{1}{\beta_0+2}}\nonumber\\
&\leq& \frac12\beta_0R^2\int_{B_{2R}}
\phi^{2}f^{\beta_0+2}+C_{17}\beta_0^{\beta_0+3}R^{-2(\beta_0+1)}|B_{2R}|.
\end{eqnarray}
Substituting the estimates \eqref{Qeq1} and \eqref{Qeq2} into \eqref{Qeq0}, we obtain
\begin{eqnarray*}\label{Qeq3}
\left(\int_{B_{2R}}\phi^{2\chi} f^{(\beta_0+1)\chi}
\right)^{\frac{1}{\chi}} \leq
2e^{C_{11}\beta_0}C_{13}^{\beta_0+1}\beta_0^3\left(\frac{\beta_0}{R}\right)^{2(\beta_0+1)}|B_{2R}|^{1-\frac2\nu}.\end{eqnarray*}
Taking the $(\beta_0+1)$-st root on both sides, we get
\begin{eqnarray*}\label{Qeq4} \|f\|_{L^{\beta_1}(B_{\frac32R})} \leq C_{18}\left(\frac{\beta_0}{R}\right)^{2}|B_{2R}|^{\frac{1}{\beta_1}}.\end{eqnarray*}
\end{proof}

Now we start from \eqref{Peq15} and use Moser's iteration to prove
Theorem \ref{main thm}.

Let $R_k=R+R/2^k$ and $\phi_k\in \Li_0(B_{R_k})$ satisfy
\begin{eqnarray*}\label{Peq16}
0\leq \phi_k\leq 1,\quad \phi_k\equiv 1 \hbox{ in }
B_{R_{k+1}},\quad|\nabla\phi_k|\leq
C\frac{2^{k+1}}{R}.\end{eqnarray*} Let $\beta_0,\beta_1$ be the
numbers in Lemma \ref{lem1} and $\beta_{k+1}=\beta_k\chi$ for $k\geq
1.$ One can deduce from \eqref{Peq15}  with $\beta+1=\beta_k$ and
$\phi=\phi_k$ that (we have dropped the last term in the right-hand
side of \eqref{Peq15} since it is negative)
\begin{eqnarray*}\label{Peq17} \|f\|_{L^{\beta_{k+1}}(B_{R_{k+1}})}\leq e^{C_{19}\frac{\beta_0}{\beta_k}}|B_{2R}|^{-\frac2\nu\frac{1}{\beta_k}}(4^{k}+\beta_0^2\beta_k)^{\frac{1}{\beta_k}}\|f\|_{L^{\beta_{k}}(B_{R_{k}})}.\end{eqnarray*}
Hence by iteration we get
\begin{eqnarray*}\label{Peq18} \|f\|_{L^\infty(B_{R})}\leq  e^{C_{19}\beta_0\sum_k \frac{1}{\beta_k}}|B_{2R}|^{-\frac2\nu \sum_k \frac{1}{\beta_k}}\prod_{k}(4^{k}+2\beta_0^3\chi^k)^{\frac{1}{\beta_k}}\|f\|_{L^{\beta_{1}}(B_{\frac32 R})}. \end{eqnarray*}
Since $\sum_k \frac{1}{\beta_k}=\frac{\nu}{2}\frac{1}{\beta_1}$ and
$\sum_k \frac{k}{\beta_k}$ converges, we have
\begin{eqnarray*}\label{Peq19} \|f\|_{L^\infty(B_{R})}&\leq&  C_{20}e^{C_{21}\frac{\beta_0}{\beta_1}}\beta_0^{\frac{3\nu}{2}\frac{1}{\beta_1}}|B_{2R}|^{-\frac{1}{\beta_1}}\|f\|_{L^{\beta_{1}}(B_{\frac32 R})}\nonumber\\&\leq &C_{22}|B_{2R}|^{-\frac{1}{\beta_1}}\|f\|_{L^{\beta_{1}}(B_{\frac32 R})}.\end{eqnarray*}
Using Lemma \ref{lem1}, we conclude
\begin{eqnarray*}\label{Peq20} \|f\|_{L^\infty(B_{R})}\leq C(n)\frac{(1+\sqrt{K}R)^2}{R^2},\end{eqnarray*}
which implies
\begin{eqnarray*}\label{Peq21} \|\nabla \log u\|_{L^\infty(B_{R})}\leq C(n)\frac{1+\sqrt{K}R}{R}.\end{eqnarray*}
This proves Theorem \ref{main thm}.

\

{\bf Acknowledgements.} The research leading to these
results has received funding from the European Research Council
under the European Union's Seventh Framework Programme
(FP7/2007-2013) / ERC grant agreement n$^\circ$~267087. We would like to thank Prof. Xi-Ping Zhu and Dr. Hui-Chun Zhang for their valuable suggestions on the Bochner formula in \cite{ZZ3}. We also thank the anonymous
referee for his/her useful comments and suggestions. 


\begin{thebibliography}{99}

\bibitem{BBI} D. Burago, Yu. Burago and S. Ivanov, A course in metric geometry, Graduate Studies in Mathematics 33,
American Mathematical Society, Providence, RI, 2001.

\bibitem{BGP} Yu. Burago, M. Gromov and G. Perelman, A. D. Aleksandrov spaces with curvatures bounded
below, Russian Math. Surveys 47 (1992), no. 2, 1--58.

\bibitem{C} J. Cheeger, Differentiability of Lipschitz functions on metric
 measure spaces, Geom. Funct. Anal. 9 (1999), no. 3, 428--517.

\bibitem{CY} S. Y. Cheng and S. T. Yau, Differential equations on Riemannian manifolds and their geometric applications, Comm. Pure Appl. Math. 28 (1975), no.3, 333--354 .

\bibitem{HK} P. Hajlasz and P. Koskela, Sobolev meets Poincare, C. R. Acad Sci. Paris
320 (1995), no. 10, 1211--1215.

\bibitem{Hua} B. Hua, Harmonic Functions of Polynomial Growth on Singular Spaces with Non-negative Ricci Curvature,Proc. Amer. Math. Soc.
139 (2011), no. 6, 2195--2205.


\bibitem{KMS} K. Kuwae, Y. Machigashira and T. Shioya, Sobolev spaces, Laplacian and heat kernel on Alexandrov spaces, Math. Z. 238
(2001), no.2, 269--316.

\bibitem{KS1} K. Kuwae and T. Shioya, On generalized measure contraction property
and energy functionals over Lipschitz maps, ICPA98 (Hammamet),
Potential Anal. 15 (2001), no. 1--2, 105--121.

\bibitem{KS2} K. Kuwae and T. Shioya, Sobolev and Dirichlet spaces over maps
between metric spaces, J. Reine Angew. Math. 555 (2003), 39--75.

\bibitem{KS3} K. Kuwae and T. Shioya, Laplacian comparison for Alexandrov
spaces, arXiv:0709.0788.

\bibitem{KS4} K. Kuwae and T. Shioya, Infinitesimal Bishop-Gromov
condition for Alexandrov spaces, Probabilistic approach to geometry,
293--302, Adv. Stud. Pure Math., 57, Math. Soc. Japan, Tokyo, 2010.



\bibitem{LV1} J. Lott and C. Villani, Ricci curvature for metric-measure spaces via optimal transport, Ann. of Math. (2) 169 (2009), no. 3, 903--991.

\bibitem{LV2} J. Lott and C. Villani, Weak curvature bounds and functional inequalities, J. Funct. Anal. 245 (2007), no. 1, 311--333.

\bibitem{MW}  O. Munteanu and J. Wang, Smooth metric measure spaces with non-negative curvature. Comm. Anal. Geom. 19  (2011), no. 3, 451-- 486.






\bibitem{O1} S. Ohta, On the measure contraction property of metric measure spaces, Comment. Math. Helv. 82
(2007), no. 4, 805--828.

\bibitem{P} A. Petrunin, Subharmonic functions on Alexandrov space, preprint (1996), available online at www.math.psu.edu/petrunin/.


\bibitem{SC} L. Saloff-Coste, Aspects of Sobolev-type inequalities, London Math. Soc. Lecture Note Ser. 289. Cambridge University Press, Cambridge, 2002.

\bibitem{St} K. Sturm, On the geometry of metric measure spaces, II, Acta Math. 196
(2006), no. 1, 133--177.



\bibitem{WZ} X. Wang and L. Zhang, Local gradient estimate for $p$-harmonic functions on Riemannian manifolds, Comm. Anal. Geom. 19 (2011), no. 4, 759--771.



\bibitem{X} C. Xia, Local gradient estimate for harmonic functions on Finsler manifolds, preprint.

\bibitem{Y} S. T. Yau, Harmonic functions on complete Riemannian manifolds, Comm. Pure Appl. Math. 28 (1975), 201--228.


\bibitem{ZZ1} H.-C. Zhang and X.-P. Zhu, Ricci curvature on Alexandrov
spaces and rigidity theorems, Comm. Anal. Geom. 18(2010), no. 3,
503--553.

\bibitem{ZZ2} H.-C. Zhang and X.-P. Zhu, On a new definition of Ricci
curvature on Alexandrov spaces, Acta Math. Sci. Ser. B Engl. Ed. 30
(2010), no. 6, 1949--1974.

\bibitem{ZZ3} H.-C. Zhang and X.-P. Zhu, Yau's gradient estimates on Alexandrov spaces,  J. Differential Geom. 91 (2012), no. 3, 445--522.

\end{thebibliography}
\end{document}